\documentclass[runningheads]{llncs}
\usepackage{graphicx}
\usepackage{epsfig}
\usepackage{graphics}
\usepackage{array}
\usepackage{amsmath}
\usepackage{algorithm}
\usepackage{epstopdf}
\usepackage{algorithm}
\usepackage{algorithmic}
\usepackage{amssymb}
\usepackage{fullpage}
\date{}

\newtheorem{dfn}{Definition}
\title{Constrained Hitting Set and Steiner Tree in $SC_k$ and $2K_2$-free Graphs}
\author{S.Dhanalakshmi and N.Sadagopan} 
\institute{Indian Institute of Information Technology, Design and Manufacturing, Kancheepuram, Chennai, India. \\
\email{$\{mat12d001, sadagopan\}@iiitdm.ac.in$}}
\begin{document}
\maketitle

\begin{abstract}
\emph{Strictly Chordality-$k$ graphs ($SC_k$)} are graphs which are either cycle-free or every induced cycle is of length exactly $k, k \geq 3$. Strictly chordality-3 and strictly chordality-4 graphs are well known chordal and chordal bipartite graphs, respectively. For $k\geq 5$, the study has been recently initiated in \cite{sadagopan} and various structural and algorithmic results are reported. In this paper, we show that maximum independent set (MIS), minimum vertex cover, minimum dominating set, feedback vertex set (FVS), odd cycle transversal (OCT), even cycle transversal (ECT) and Steiner tree problem are polynomial time solvable on $SC_k$ graphs, $k\geq 5$. We next consider $2K_2$-free graphs and show that FVS, OCT, ECT, Steiner tree problem are polynomial time solvable on subclasses of $2K_2$-free graphs.
\\

\noindent \textbf{Keywords:} Strictly Chordality $k$ graphs, $2K_2$-free graphs, Feedback Vertex Set, Odd (Even) Cycle Transversal, Steiner tree.

\end{abstract}

\section{Introduction}

Strictly Chordality $k$ graphs ($SC_k$ graphs) are graphs which are either cycle-free or every induced cycle is of length $k$. This graph class was introduced very recently by Dhanalakshmi et al. in \cite{sadagopan} by generalizing Chordal and Chordal bipartite graphs in a larger dimension. $SC_3$ and $SC_4$ graphs are well known chordal graphs and chordal bipartite graphs, which are well studied as it helps to identify the gap between NP-Complete input instances and polynomial-time solvable input instances on many problems. Problems such as clique, independent set, coloring have polynomial-time algorithms restricted to $SC_3(SC_4)$ graphs. On the similar line, authors of \cite{sadagopan} have explored $SC_{k \geq 5}$ in detail from both structural and algorithmic front. In \cite{sadagopan}, polynomial-time algorithms for problems such as testing, Hamiltonian cycle, coloring, tree-width, and minimum fill-in have been presented.

In this paper, we revisit $SC_k$ graphs and study classical problems such as MIS, dominating set, FVS, OCT, ECT and Steiner tree. In recent times, these problems are extensively studied in the context of parameterized complexity \cite{octnp,ectnp}. Also, cycle hitting problems such as FVS, OCT, ECT have polynomial-time algorithms restricted to chordal and chordal bipartite graphs \cite{kloks}. Further, independent set and vertex cover also have polynomial-time algorithms in chordal \cite{gavril} and chordal bipartite graphs. Steiner tree, a generalization of classical minimum spanning tree problem and dominating set are known to be NP-Complete in chordal and chordal bipartite graphs.

It is important to highlight that chordal (chordal bipartite) graphs have a special ordering, on vertices namely \emph{perfect vertex elimination ordering} (\emph{perfect edge elimination ordering}) and this ordering is greatly used in solving all of the above combinatorial problems. For $SC_{k\geq 5}$ graphs, a \emph{vertex cycle ordering (VCO)} is proposed in \cite{sadagopan}. It would be an interesting attempt to see whether VCO helps in solving the above mentioned combinatorial problems restricted to $SC_k$ graphs. This is the first focus of this paper.

The second focus of this paper is to study subclasses of $2K_2$-free graphs from minimal vertex separator (MVS) perspective and analyze the complexity of cycle hitting problems in $2K_2$-free graphs. $2K_2$-free graphs have received good attention in the literature as it is a subclass of $P_5$-free graphs and a superclass of split graphs. Interestingly, Steiner tree \cite{white} and Dominating set \cite{white} is NP-Complete on $2K_2$-free graphs and other classical problems are polynomial-time solvable \cite{chung,hujter,meister}. In this paper, we investigate the complexity of cycle hitting problems and Steiner tree on subclasses of $2K_2$-free graphs and present polynomial-time algorithms for all of them.

%
%

\noindent \textbf{Organization of the paper:} In Section \ref{sec::defn}, we introduce basic terminologies and theorems used in this paper. The algorithmic results on $SC_k$ graphs; maximum independent set, odd (even) cycle transversal, feedback vertex set, dominating set and Steiner tree are presented in Section \ref{sec::algs}. In Section \ref{2k2}, we present the structural and algorithmic results on the subclasses of $2K_2$-free graphs.


\section{Preliminaries}
\label{sec::defn}

\subsection{Graph Preliminaries}
We follow the notation as in \cite{dbwest,golumbic}. Let $G$ be a simple, connected and undirected graph with the non-empty vertex set $V(G)$ and the edge set $E(G)$= \{\{$u,v$\} $\vert$ $u,v \in V(G)$ and $u$ is adjacent to $v$ in $G$ and $u \neq v$\}. The $neighborhood$ of a vertex $v$ of $G$, $N_G$($v$), is the set of vertices adjacent to $v$ in $G$. The degree of the vertex $v$ is $d_G(v) = \vert N_G(v) \vert$. Let $S\subset V(G)$, we define $N_G(S)$ as $\{u\in V(G)\vert ~ \forall ~v \in S, \{u,v\}\in E(G)\}$. A \emph{cycle} $C$ on $n$-vertices is denoted as $C_n$, where $V(C) = \{x_1, x_2, \ldots, x_n\}$ and $E(C) = \{\{x_1, x_2\}, \{x_2,x_3\}, \ldots, \{x_{n-1},x_n\}, \{x_n,x_1\}\}$. The graph $G$ is said to be $connected$ if every pair of vertices in $G$ has a path and if the graph is not connected it can be divided into disjoint connected $components$ $G_1, G_2, \ldots, G_k$, $k \geq 2$, where $V(G_i)$ denotes the set of vertices in the component $G_i$. The graph $G$ is said to be \emph{k-connected} (or \emph{k-vertex connected}) if there does not exist a set of $k-1$ vertices whose removal disconnects the graph. The graph $M$ is called a $subgraph$ of $G$ if $V(M)$ $\subseteq$ $V(G)$ and $E(M)\subseteq E(G)$. The subgraph $M$ of a graph $G$ is said to be $induced$ $subgraph$, if for every pair of vertices $u$ and $v$ of $M$, \{$u,v$\} $\in$ $E(M)$ if and only if \{$u,v$\} $\in$ $E(G)$ and it is denoted by $[M]$. An $induced$ $cycle$ is a cycle that is an induced subgraph of $G$. The graph $G$ is said to be \emph{cycle free} if there is no induced cycle in $G$. 

\subsection{Definitions and properties on $SC_k$ graphs}

\begin{theorem}{\cite{sadagopan}}
\label{thm::sckcons}
 A graph $G$ is a $SC_k$ graph if and only if it can be constructed iteratively by any one of the following operations.
\begin{itemize}
\item[(i)] $K_1$ is an $SC_k$ graph.
\item[(ii)] $C_k$ is an $SC_k$ graph.
\item[(iii)] If $G$ is an $SC_k$ graph, then the graph $G'$, where, $V(G') = V(G) \cup \{v\}$, $E(G') = E(G) \cup \{u, v\}$ such that $v \notin V(G)$ and $u$ is any vertex in $V(G)$,  is also an $SC_k$ graph.
\item[(iv)] If $G$ is an $SC_k$ graph, then the graph $G'$, where, $V(G') = V(G) \cup \{v_1, v_2, \ldots, v_{k-1}\}$, $E(G') = E(G) \cup \{\{u, v_1\}, \{v_1, v_2\}, \{v_2, v_3\},  \ldots, \{v_{k-2}, v_{k-1}\}, \{v_{k-1}, u\} \}$ such that $\{v_1, v_2, \ldots, v_{k-1}\} \cap V(G) = \phi$ and $u$ is any vertex in $V(G)$,  is also an $SC_k$ graph.
\item[(v)] If $G$ is an $SC_k$ graph, then the graph $G'$, where, $V(G') = V(G) \cup \{v_1, v_2, \ldots, v_{k-2}\}$, $E(G') = E(G) \cup \{\{u, v_1\}, \{v_1, v_2\}, \{v_2, v_3\}, \ldots, \{v_{k-2}, v\} \}$ such that $\{v_1, v_2, \ldots, v_{k-2}\} \cap V(G) = \phi$ and $\{u, v\}$ is any edge in $E(G)$,  is also an $SC_k$ graph.
\item[(vi)] If $G$ is an $SC_k$ graph and $ k=2m+4, m \geq 1$, then the graph $G'$, where, $V(G') = V(G) \cup \{v_1, v_2, \ldots, v_{\frac{k}{2}-1}\}$, $E(G') = E(G) \cup \{\{u_1, v_1\}, \{v_1, v_2\}, \{v_2, v_3\}, \ldots, \{v_{\frac{k}{2}-1}, u_{\frac{k}{2}+1}\} \}$ such that $\{v_1, v_2, \ldots, v_{\frac{k}{2}-1}\} \cap V(G) = \phi$ and $\{u_1,u_2,\ldots,u_{\frac{k}{2}+1}\}$ is any path of length $\frac{k}{2}+1$ contained in no induced cycle in $G$ or in any one induced cycle $S_i$ of length $k$ in $G$ such that there does not exist an induced cycle $S_j$ in $G$ with $V(S_i) \cap V(S_j) = \{w_1,\ldots,w_{\frac{k}{2}+1}\}$, $w_p = u_p$ for some $p \in\{1,\ldots, \frac{k}{2}+1\}$ and for at least one $q \in\{1,\ldots, \frac{k}{2}+1\}$, $w_q \neq u_q$.
\end{itemize}
\end{theorem}

Throughout this subsection, the graph $G$ refers to an $SC_k$ graph, $k \geq 5$. It is clear from the above bi-implication that we can get a Vertex Cycle Ordering (VCO) for any $G$ in at most $n$ iterations, where $n$ is the number of vertices in $G$ \cite{sadagopan}.

\begin{definition}
 Let $\mu = (x_1,\ldots, x_s)$, $1\leq s \leq n$, be the ordering of $G$. If $s=1$, then either $G$ is a trivial graph or a cycle of length $k$. If $s \geq 2$, then the label($x_i$), $i<s$, denotes the\\
\noindent (a) pendant vertex if it satisfies the condition (iii) of \emph{Theorem \ref{thm::sckcons}},\\
\noindent (b) 0-pendant cycle if it satisfies the condition (iv) of \emph{Theorem \ref{thm::sckcons}} and if $u$ is not a part of any cycle in $G$, \\
\noindent (c) 1-pendant cycle if it satisfies the condition (iv) of \emph{Theorem \ref{thm::sckcons}} and if $u$ is part of at least one cycle, \\
\noindent (d) 2-pendant cycle if it satisfies the condition (v) of \emph{Theorem \ref{thm::sckcons}} and\\ 
\noindent (e) $(\frac{k}{2}+1)$-pendant cycle if it satisfies the condition (vi) of \emph{Theorem \ref{thm::sckcons}} w.r.t the induced graph on $(x_i, x_{i+1}, \ldots, x_s)$. Note that, in a $(\frac{k}{2}+1)$-pendant cycle $S$, $S$ can have either $u_1$ or $u_{\frac{k}{2}+1}$ as a cut vertex but not both. 
\end{definition}

\begin{definition}
 A graph $G$ is said to be a cage graph of size $n$ denoted as $CAGE(n,l)$ if there exist $w,z \in V (G)$ such that $\{w,u^{i}_{1}\},\{z,u^{i}_{l-2}\} \in E(G)$ for all $1 \leq i \leq n$ and $P^{i}_{u_1u_{l−2}}$ is a path of length $l-2$.   
\end{definition}

\begin{figure}[H]
\centering
\includegraphics[height=3.5cm, width=4.5cm]{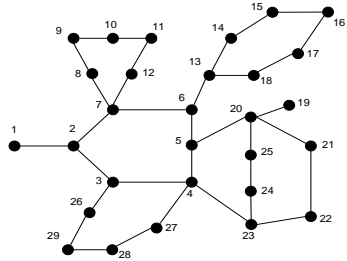}
\caption{An example for an $SC_6$ graph. One of the vertex cycle ordering for this graph is ($\{1\}$, $\{19\}$, $\{13,14,15,16,17,18\}$, $\{13\}$, $\{7,8,9,10,11,12\}$, $\{3,4,27,28,29,26\},$ $\{20,21,22,23,24,25\}$, $\{4,5,20,25,24,23\}$, $\{2,3,4,5,6,7\}$), where the vertices $1$ and $19$ are said to be pendant, $(13,14,15,16,17,18)$ is a 0-pendant cycle, $(7,8,9,10,11,12)$ is a 1-pendant cycle, $(3,4,27,28,29,26)$ as 2-pendant cycle and $(20,21,22,23,24,25)$ is a $4$-pendant cycle. The graph induced on the vertex set $\{4,5,20,21,22,23,24,25\}$ is the $CAGE(3,4)$.}
\label{fig::vco}
\end{figure}

\subsection{Definitions and properties on $2K_2$-free graphs}

\begin{lemma}{\cite{mano}}
 A connected graph is $2K_2$ free if and only if it forbids $H_{1}$, $H_{2}$ and  $H_{3}$ as an induced subgraphs.
\begin{figure}[H]
\begin{center}
\includegraphics[scale=0.3]{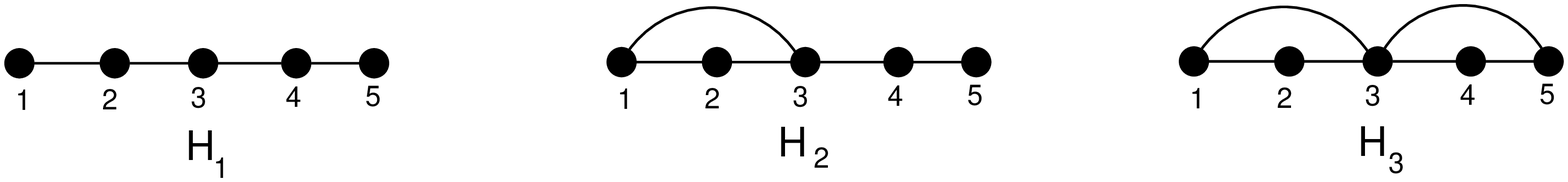} 
\end{center}
\end{figure}
\end{lemma}

\begin{dfn}
Let $G$ be a graph and $S \subset V(G)$. A vertex $v \in V(G\backslash S)$ is said to be a \emph{universal vertex} if $\forall ~ x \in S, \{x, v\} \in E(G)$. An edge $\{u, v\}$ is said to be a \emph{universal edge} if $\forall ~x \in S$, either $\{x, u\} \in E(G)$ or $\{x, v\} \in E(G)$.
\end{dfn}

\begin{theorem}{\cite{mano}}
\label{2k2pre}
Let $G$ be a connected graph and $S$ be any minimal vertex separator of $G$. Let $G_1, G_2, \ldots , G_l$, $(l \geq 2)$ be the connected components in $G\backslash S$. $G$ is $2K_2$ free if and only if it satisfies the following conditions:
\begin{itemize}
\item[(i)] $G\backslash S$ contains at most one non-trivial component. Further, if $G\backslash S$ has a non-trivial component, say $G_1$, then the graph induced on $V(G_1)$ does not contain $H_1$, $H_2$, $H_3$ as an induced subgraphs.
\item[(ii)] Every trivial component of $G\backslash S$ is universal to $S$.
\item[(iii)] Every edge in the non-trivial component of $G\backslash S$ is universal to $S$.
\item[(iv)] The graph induced on $V(S)$ is either connected or has at most one non-trivial component. Further, if the graph induced on $V(S)$ has a non-trivial component, say $S_1$, then the graph induced on $V(S_1)$ does not contain $H_1$, $H_2$, $H_3$ as an induced subgraphs.
\item[(v)] If $S$ and $G\backslash S$ has a non-trivial component, say $S_1$ and $G_1$, respectively, then every edge in $S_1$ is universal to $G_1\backslash M$, where $M = \{v \in V(G_1) \mid N_G(v) \cap V (S) = \phi\}$.
\end{itemize}
\end{theorem}
\section{Algorithmic Results on $SC_k$ graphs}
\label{sec::algs}

Let $G$ be a strictly chordality $k$ graph, $k\geq 5$, and let $\mu = (x_1, \ldots,x_s)$ be the VCO of $G$, $1 \leq s \leq n$. Each algorithm makes use of a VCO and picks the \emph{desired} vertices. At every stage of the algorithm, \emph{pruning} of undesired vertices is also done. Our algorithms are based on dynamic programming paradigm. 

\noindent For each $x_i$, $1\leq i \leq s$, we define \emph{label(x$_i$)} that denotes the associated vertices in $x_i$. For Figure \ref{fig::vco}, $\mu = (x_1,\ldots, x_9)$, where $label(x_1)=\{1\}$, $label(x_3)=\{13,14,15,16,17,18\},\ldots$,  $label(x_9)=\{2,3,4,5,6,7\}$. \\

\noindent \textbf{Problem 1} \emph{Maximum Independent Set (MIS)}.

Given an $SC_k$ graph $G$, $k\geq 5$, an independent set $S\subseteq V(G)$ such that $\forall ~ u,v \in S, ~ u,v \notin E(G)$. The objective is to find an independent set in $G$ of maximum cardinality. We now present an algorithm to find a MIS. 

\begin{itemize}
\item[1.] Let $\mu = (x_1,\ldots, x_s)$, $1 \leq s \leq n$, be the VCO of $G$
\item[2.] Find an MIS $S'$ for $label(x_1)$. Add $S'$ to $S$.
\item[3.] Remove $S'\cup N_G(S')$ from $G$ and let the resulting graph be $G'$.
\item[4.] Update $\mu$ and repeat \emph{Steps 2} and $3$.   
\end{itemize}
Let $I(G)$ denote the independent set of $G$ with maximum size. Then, $
I(G) = I(label(x_1)) \cup I(G\backslash M)$ where, $M= S'\cup N_G(S')$.\\

\noindent \textbf{Computing $I(label(x_1))$:}
\begin{lemma}
\label{x1pendvertex}
$I(label(x_1)) = \{u\}$ if $label(x_1)=\{u\}$ is a pendant vertex.
\end{lemma}

\vspace{-0.4cm}

\begin{proof}
On the contrary, assume that $u$ is not a part of any maximum independent set of $G$. Since $u$ is a pendant vertex, $N_G(u)$ is a singleton set, say $\{v\}$. If $v$ is also a pendant vertex, then there is nothing to prove. Assume that $v$ is not a pendant vertex. It is clear from the definition of $I(G)$ that either $u \in I(G)$ or $v \in I(G)$. By our assumption, $u \notin I(G)$. Thus, $v \in I(G)$. By choosing $v$, we are forced not to add the vertices in $N_G(v)$, whose cardinality is strictly greater than zero. This will contradict the maximality of $I(G)$ unless $G$ is either $P_{2m}, m \geq 2$ or $\vert P_{ux} \vert \geq 2m-1, m \geq 2$ where $x$ i
s the first vertex of degree at least three in $G$. 
$\hfill \qed$
\end{proof}

\begin{lemma}
\label{x10/1pendcycle}
Let $label(x_1)=\{u_1,\ldots,u_k\}$ be the 0-pendant cycle (or 1-pendant cycle) where $deg_G(u_1)\geq 3$, $\{u_1,u_k\} \in E(G)$ and $\{u_i,u_{i+1}\}\in E(G), 1 \leq i \leq k-1$. Then $I(label(x_1)) = \{u_2, u_4, \ldots, u_{k-1}\}$ if $k$ is odd and $I(label(x_1)) = \{u_2, u_4, \ldots, u_{k}\}$ if $k$ is even.
\end{lemma}

\begin{proof}
It is clear that, the maximum size of an independent set of a cycle $C_k$ is $\lfloor \frac{k}{2} \rfloor$. The cardinality of the given set $I(label(x_1))$ is $\lfloor \frac{k}{2} \rfloor$. Thus, $I(label(x_1))$ is the maximum independent set of $label(x_1)$. It remains to show that the set $I(label(x_1))$ does not affect the maximality of $I(G)$. i.e., to prove that the maximality of $I(G)$ is affected if we choose $I(label(x_1))=\{u_1,u_3,\ldots,u_{k-2}\}$ when $k$ is odd and $I(label(x_1))=\{u_1,u_3,\ldots,u_{k-1}\}$ when $k$ is even. It is enough to prove that $u_1$ is not part of $I(G)$. Since $deg_G(u_1)\geq 3$, any MIS $I'$ containing $u_1$ has the property that $I' <I$. Thus, if we choose $u_1$ for $I(label(x_1))$, then the cardinality of the resultant independent set for $G$ is either $\vert I(G)\vert $ or less than $\vert I(G)\vert$. $\hfill \qed$
\end{proof}

\begin{lemma}
\label{x12pendcycle}
Let $label(x_1)=\{u_1,\ldots,u_k\}$ be the 2-pendant cycle where $\{u_1,u_k\} \in E(G)$ and $\{u_i,u_{i+1}\}\in E(G), 1 \leq i \leq k-1$,  $deg_G(u_1)\geq 3$ and $deg_G(u_2)\geq 3$. Then $I(label(x_1))=\{u_3,u_5,\ldots,u_k\}$ if $k$ is odd and $I(G) = \max \{I_1(label(x_1))\cup I(G\backslash M_1), I_2(label(x_1))\cup I(G\backslash M_2), I_3(label(x_1)) \cup I(G\backslash M_3)\}$ if $k$ is even, where $I_1(label(x_1)) = \{u_1,u_3,\ldots,u_{k-1}\}$, $I_2(label(x_1)) = \{u_2,u_4,\ldots,u_{k}\}$, $I_3(label(x_1))=\{u_3,\ldots,u_{k-1}\}$ and $M_i=\bigcup\limits_{u \in I_i(label(x_1))}(u \cup N_G(u))$, $i \in \{1,2,3\}$.
\end{lemma}

\begin{proof}
We prove this lemma by splitting $k$ into odd and even. \textbf{Case 1:} When $k$ is odd. The size of the set $I(label(x_1))$ is $\lfloor \frac{k}{2} \rfloor$, which is the maximum size of an independent set in an odd cycle of length $k$. An argument similar to \emph{Lemma \ref{x10/1pendcycle}} proves that the set $I(label(x_1))$ does not affect the maximality of $I(G)$. \textbf{Case 2:} When $k$ is even. The size of both the sets $I_1(label(x_1))$ and $I_2(label(x_1))$ are $\frac{k}{2}$, which is the maximum size of an independent set in an even cycle of length $k$. In order to get the maximum independent set for $G$, the maximum is taken over $I_i(label(x_1))\cup I(G\backslash M_i)$, $i=1,2,3$, and the conclusion follows. We consider $I_3(label(x_1))$, not to contradict the maximality of $I(G)$ due to the presence of both $u_1$ and $u_2$. 
$\hfill \qed$
\end{proof}

\begin{lemma}
\label{x1k/2+1pendcycle}
Let $label(x_1)=\{u_1,\ldots,u_k\}$ be the $(\frac{k}{2}+1)$-pendant cycle where $\{u_1,u_k\}\in E(G)$, $\{u_i,u_{i+1}\}\in E(G), 1 \leq i \leq k-1$,  $deg_G(u_1)\geq 3$ and $deg_G(u_{\frac{k}{2}+1})\geq 3$. Then  $I(label(x_1))=\{u_2,u_4,\ldots,u_k\}$ if $k=4m+4, m \in \mathbb{N}$ and $I(G) = \max \{I_1(label(x_1))\cup I(G\backslash M_1), I_2(label(x_1))\cup I(G\backslash M_2)\}$ if $k=4m+2, m \in \mathbb{N}$, where $I_1(label(x_1)) = \{u_1,u_3,\ldots,u_{k-1}\}$, $I_2(label(x_1)) = \{u_2,u_4,\ldots,u_{k}\}$ and $M_i = \bigcup\limits_{u \in I_i(label(x_1))}(\{u\} \cup N_G(u))$, $i=1,2$.
\end{lemma}

\begin{proof}
The $(\frac{k}{2}+1)$-pendant cycle forms a $CAGE(p,\frac{k}{2}+1)$, $p \geq 3$. It is clear from the definition of $(\frac{k}{2}+1)$-pendant cycle that either $u_1$ is a cut vertex or $u_{\frac{k}{2}+1}$ is a cut vertex but not both and the degree of each vertices in the set $\{u_2,\ldots,u_{\frac{k}{2}},u_{\frac{k}{2}+2},\ldots,u_k\}$ is two. We prove this lemma by partitioning the $k$ into the following two cases: \textbf{Case 1:} $k=4m+4, m \in \mathbb{N}$. The size of the set $I(label(x_1))=\{u_2,u_4,\ldots,u_k\}$ is $\frac{k}{2}$, which is maximum. Moreover, the set does not include $u_1$ and $u_{\frac{k}{2}+1}$ and this concludes the proof of this case. \textbf{Case 2:} $k=4m+2, m \in \mathbb{N}$. The size of both $I_1(label(x_1))$ and $I_2(label(x_1))$ are $\frac{k}{2}$, which is maximum, where $I_1(label(x_1))$ is the set containing $u_1$ and $I_2(label(x_1))$ is the set containing $u_{\frac{k}{2}+1}$. By the definition of $(\frac{k}{2}+1)$-pendant cycle, it is enough to take the maximum of $I_1(label(x_1))\cup I(G\backslash M_1)$ and $I_2(label(x_1))\cup I(G\backslash M_2)$, to get $I(G)$.
$\hfill \qed$
\end{proof}

\begin{theorem}
Let $G$ be an $SC_k$ graph. A maximum independent set can be found in polynomial time. Further, a minimum vertex cover can be computed in polynomial time.
\end{theorem}

\begin{proof}
The claim follows from \emph{Lemmas \ref{x1pendvertex}-\ref{x1k/2+1pendcycle}} and the fact that VCO can be computed in polynomial time \cite{sadagopan}. A minimum vertex cover for $G$ can be obtained by taking the complement of a maximum independent set of $G$, which can be obtained in polynomial time. Thus, the theorem.  $\hfill \qed$
\end{proof}

%
%

\noindent \textbf{Problem 2} \emph{Minimum Dominating Set}.

Given an $SC_k$ graph $G$, $k\geq 5$, the objective is to find a vertex subset $S$ of $G$ with minimum cardinality such that for every $v \in V(G)$, either $v \in S$ or $v\in N_G(x)$ for some $x\in S$. 

\noindent The algorithm for a minimum dominating set: Start by finding the VCO for a given $SC_k$ graph $G$, say $\mu = (x_1,\ldots, x_s)$, $1 \leq s \leq n$. Now, find the minimum dominating set for the first element in the ordering. This immediately suggests us to remove the chosen vertices along with its neighbors from $G$ and we recursively compute the dominating set.
\begin{equation}
\nonumber
D(G) = D(label(x_1)) \cup D(G\backslash M)
\end{equation}

\noindent  where $D(G)$ denotes a dominating set of $G$ with minimum size and $M= \bigcup\limits_{u\in D(label(x_1))} (\{u\} \cup N_G(u))$\\

\noindent \textbf{Computing $D(label(x_1))$:}
\begin{lemma}
\label{dx1pendvertex}
$D(label(x_1)) = \{v\}$ if $label(x_1)=\{u\}$ is a pendant vertex and $N_G(u)=\{v\}$.
\end{lemma}

\begin{proof}
The pendant vertex $u$ can be dominated either by choosing its neighbor $v$ or by choosing the vertex $u$ itself. By choosing $v$, we can dominate more vertices in $G$, which helps us to minimize the size of the dominating set for $G$.
$\hfill \qed$
\end{proof}

\begin{lemma}
\label{dx10/1pendcycle}
Let $label(x_1)=\{u_1,\ldots,u_k\}$ be the 0-pendant cycle (or 1-pendant cycle) where $deg_G(u_1)\geq 3$, $\{u_1,u_k\} \in E(G)$ and $\{u_i,u_{i+1}\}\in E(G), 1 \leq i \leq k-1$. Then $D(label(x_1)) = \{u_1, u_4, u_7, \ldots, u_{p}\}$ where $k-3<p\leq k$.
\end{lemma}

\begin{proof}
It is clear that, the minimum size of a dominating set of a cycle $C_k$ is $\lceil \frac{k}{3} \rceil$. The cardinality of the given set $D(label(x_1))$ is $\lceil \frac{k}{3} \rceil$. Thus, $D(label(x_1))$ is the minimum dominating set of $x_1$ and the set does not affect the minimality of $D(G)$ as $D(label(x_1))$ contains $u_1$. This completes the proof of the lemma.
$\hfill \qed$
\end{proof}

\begin{lemma}
\label{dx12pendcycle}
Let $label(x_1)=\{u_1,\ldots,u_k\}$ be the 2-pendant cycle where $\{u_1,u_k\} \in E(G)$ and $\{u_i,u_{i+1}\}\in E(G), 1 \leq i \leq k-1$, $deg_G(u_1)\geq 3$ and $deg_G(u_2)\geq 3$. Then $D(G) = \min\limits_{i=1,2} \{D_i(label(x_1))\cup D(G\backslash M_1)\}$, where $D_1(label(x_1)) = \{u_1,u_4,\ldots,u_{p}\}$, $D_2(label(x_1)) = \{u_2,u_5,\ldots,u_{p'}\}$, $M_i=\bigcup\limits_{u \in D_i(label(x_1))}(\{u\} \cup N_G(u))$, $i \in \{1,2\}$, $k-3<p, p' \leq k$.
\end{lemma}

\begin{proof}
The size of both the sets $D_1(label(x_1))$ and $D_2(label(x_1))$ are $\lceil\frac{k}{3}\rceil$, which is the minimum dominating set in a cycle of length $k$. In order to get the minimum dominating set for $G$, the minimum is taken over $D_i(label(x_1))\cup D(G\backslash M_i)$, $i=1,2$, and the conclusion follows.
$\hfill \qed$
\end{proof}

\begin{lemma}
\label{dx1k/2+1pendcycle}
Let $label(x_1)=\{u_1,\ldots,u_k\}$ be the $(\frac{k}{2}+1)$-pendant cycle where $\{u_1,u_k\}\in E(G)$, $\{u_i,u_{i+1}\}\in E(G), 1 \leq i \leq k-1$,$deg_G(u_1)\geq 3$ and $deg_G(u_{\frac{k}{2}+1})\geq 3$. Then $D(G) = \min\limits_{i=1,2} \{D_i(label(x_1))\cup I(G\backslash M_i)\}$ where $D_1(x_1)$ is the minimum dominating set for $x_1$ including $u_1$, $D_2(x_1)$ is the minimum dominating set for $x_1$ including $u_{\frac{k}{2}+1}$, $M_i=\bigcup\limits_{u \in D_i(label(x_1))}(\{u\} \cup N_G(u))$, $i \in \{1,2\}$.
\end{lemma}

\begin{proof}
The argument similar to \emph{Lemma \ref{dx12pendcycle}} establishes the claim. $\hfill \qed$
\end{proof}

Thus, we get a polynomial-time algorithm to find a minimum dominating set using \emph{Lemmas \ref{dx1pendvertex}-\ref{dx1k/2+1pendcycle}}. \\

\noindent \textbf{Problem 3} \emph{Odd Cycle Transversal}.

Given an $SC_k$ graph $G$, $k\geq 5$, the objective is to find a vertex subset $S$ of $G$ with minimum cardinality such that $G\backslash S$ is a bipartite graph (every induced cycle is even). Since the $SC_k$ graphs does not contain a odd cycle when $k$ is even, the set $S$ is empty in this case. Hence, our problem is to find the set $S$ for $SC_{2k+1}$ graph, $k \geq 1$. Let $\mu=(x_1,\ldots,x_s)$, $1 \leq s \leq n$, be the VCO of $G$. Thus, the recursive solution is:\\ 

$OCT(G)=\begin{cases}
OCT(G\backslash \{label(x_1)\}) & \text{if }label(x_1) \text{ is a pendant vertex}\\
\{u\} \cup OCT(G\backslash \{label(x_1)\}) & \text{if }label(x_1) \text{ is a 0(1)-
pendant cycle }\\
& \text{where }deg_G(u)\geq 3, u \in label(x_1)\\
\min\{\{u\} \cup OCT(G\backslash \{label(x_1)\}), & \text{if }x_1 \text{ is a 2-pendant cycle where }\\
~~~~~~~ \{v\} \cup OCT(G\backslash \{label(x_1)\})\} & \{u,v\}\in E(G) \text{ and, } deg_G(u)\geq 3 \\
& \text{ and }deg_G(v)\geq 3, u,v \in label(x_1)
\end{cases}$
\\ 
where, $OCT(G)$ is the required set $S$.\\

\noindent \textbf{Problem 4} \emph{Even Cycle Transversal}.

Given an $SC_k$ graph $G$, $k\geq 5$, the objective is to find a vertex subset $S$ of $G$ with minimum cardinality such that $G\backslash S$ is a graph where every induced cycle is of odd length. Since the $SC_k$ graphs does not contain an even cycle when $k$ is odd, the set $S$ is empty in this case. Let $\mu=(x_1,\ldots,x_s)$, $1 \leq s \leq n$, be the VCO of $G$. Thus, the recursive solution is:\\

$ECT(G)=\begin{cases}
ECT(G\backslash \{label(x_1)\}) & \text{if }label(x_1) \text{ is a pendant vertex}\\
\{u\} \cup ECT(G\backslash \{label(x_1)\}) & \text{if }label(x_1) \text{ is a 0(1)-
pendant cycle where }\\
& deg_G(u)\geq 3, u \in label(x_1)\\
\min\{\{u\} \cup ECT(G\backslash \{label(x_1)\}),  & \text{if }label(x_1) \text{ is a 2-pendant cycle where }\\
~~~~~~~ \{v\} \cup ECT(G\backslash \{label(x_1)\})\} & \{u,v\}\in E(G) \text{ and, } deg_G(u)\geq 3 \\
& \text{ and }deg_G(v)\geq 3, u,v \in label(x_1) \\
\min\{\{u\} \cup ECT(G\backslash \{label(x_1)\}), & \text{if }label(x_1) \text{ is a }(\frac{k}{2}+1)\text{-pendant cycle where }\\
~~~~~~~\{w\} \cup ECT(G\backslash \{label(x_1)\})\} & deg_G(u)\geq 3 \text{ and }deg_G(w)\geq 3, u,w\in label(x_1)
\end{cases}$
\\ 
where, $ECT(G)$ is the required set $S$.

\begin{theorem}
$OCT(G)$ and $ECT(G)$ yield an optimum OCT and ECT, respectively. 
\end{theorem}

\begin{proof}
Arguments similar to \emph{Lemmas \ref{dx1pendvertex}-\ref{dx1k/2+1pendcycle}} establishes this claim and thus, $OCT(G)$ and $ECT(G)$ can be computed in polynomial time. $\hfill \qed$
\end{proof}

\noindent \textbf{Problem 5} \emph{Feedback Vertex Set}.

Given an $SC_k$ graph $G$, $k\geq 5$, the objective is to find a vertex subset $S$ of $G$ with minimum cardinality such that $G\backslash S$ is a forest. It is easy to see that FVS is precisely OCT when $k$ is odd, and ECT when $k$ is even. Thus, FVS can be computed in polynomial time.\\
%

\noindent \textbf{Problem 6} \emph{Steiner Tree}.

Given an $SC_k$ graph $G$, $k\geq 5$, and a terminal set $R\subseteq V(G)$, Steiner tree asks for a tree $T$ spanning the terminal set. The objective is to minimize the number of additional vertices ($S\subseteq V(G)\backslash R$, also known as Steiner vertices).

\begin{definition}
Let $S_i$ be the $s$-pendant cycle in $G$ such that there exist a cycle $S_j$ in $G$, where either $\vert E(S_i) \cap E(S_j) \vert = 0$ or $s-1$ or $\vert V(S_i) \cap V(S_j) \vert = s$. Let $R = V(S_i) \backslash (V(S_i) \cap V(S_j))$. The \emph{removal of a s-pendant cycle} $S_i$ from $G$ yields the induced subgraph $G \backslash R$. Note that for each $S_i$, there is a corresponding $R$ and $G\backslash S_i$ corresponds to the graph $G\backslash R$. 
\end{definition}

\noindent We now present an algorithm to find a minimum Steiner Set.

\begin{itemize}
\item[1.] Remove all the pendant vertices and pendant cycles which do not contain any terminal vertex and update $G$. Return $G$ if $G$ is acyclic.
\item[2.] Let $\mu = (x_1,\ldots, x_s)$, $1 \leq s \leq n$, be the VCO of $G$
\item[3.] Find a Steiner set $S'$ for $label(x_1)$. Add $S'$ to $S$. A desired vertex $x'$ for the $label(x_1)$ is added to $R$.
\item[4.] Remove $label(x_1)$ from $G$ and let the resulting graph be $G'$.
\item[5.] Update $\mu$ and repeat \emph{Steps 1-4}.   
\end{itemize}

Let $ST(G,R)$ denote the vertex set of Steiner tree $T$ which spans $R\subseteq V(G)$ with a minimum number of Steiner vertices.

\begin{equation}
\nonumber
ST(G,R) = ST(G, (R \cap label(x_1))\cup \{x'\}) \cup ST(G, (R\backslash label(x_1))\cup \{x'\})
\end{equation}


\noindent \textbf{Computing $ST(G, (R \cap label(x_1))\cup \{x'\})$:}
\begin{lemma}
\label{sx1pendvertex}
If $label(x_1)=\{u\}$ is a pendant vertex, then $x' = v$ and $ST(G, (R \cap label(x_1))\cup \{x'\}) = V(P_{uv})$, where $v$ is the vertex of some $C_k$ in $G$ and the first vertex of $deg_G(v)\geq 3$ in a path from $u$ in $G$. 
\end{lemma}

\begin{proof}
We add the vertex $v$ to the terminal set because the required tree $T$ should be connected. Now, the only possible Steiner tree $T$ containing pendant vertex $u$ and $v$ is $P_{uv}$.
$\hfill \qed$
\end{proof}

\begin{lemma}
\label{sx10/1pendcycle}
Let $label(x_1)=\{u_1,\ldots,u_k\}$ be the 0-pendant cycle (or 1-pendant cycle) where $deg_G(u_1)\geq 3$, $\{u_1,u_k\} \in E(G)$ and $\{u_i,u_{i+1}\}\in E(G), 1 \leq i \leq k-1$. Let $\{r_1,\ldots,r_s\} \subseteq \{u_1,\ldots,u_k\}$ be the set of terminal vertices in $label(x_1)$. Then $x'=u_1$ and $ST(G, (R ~\cap ~label(x_1))\cup \{x'\}) = \min\limits_{0\leq i \leq s} V(P_i) \cup ST(G, (R\backslash label(x_1))\cup \{u_1\})$ where $P_i$ is the induced path obtained by removing the internal vertices of $P_{r_ir_{i+1}}$, $1\leq i \leq s-1$ from $label(x_1)$, $P_0$ and $P_s$ is obtained by removing the internal vertices of $P_{u_1r_{1}}$ and $P_{r_su_1}$ from $label(x_1)$, respectively.
\end{lemma}

\begin{proof}
We add the vertex $u_1$ to the terminal set because the required tree $T$ should be connected. The minimum of all possibilities over the $label(x_1)$ is considered to get a minimum Steiner tree.
$\hfill \qed$
\end{proof}

\begin{lemma}
\label{sx12pendcycle}
Let $label(x_1)=\{u_1,\ldots,u_k\}$ be the 2-pendant cycle where $\{u_1,u_k\} \in E(G)$ and $\{u_i,u_{i+1}\}\in E(G), 1 \leq i \leq k-1$, $deg_G(u_1)\geq 3$ and $deg_G(u_2)\geq 3$. Let $\{r_1,\ldots,r_s\}\subseteq \{u_1,\ldots,u_k\}$ be the set of terminal vertices in $label(x_1)$. Then $x'$ is either $u_1$ or $u_2$ and $ST(G, (R \cap label(x_1))\cup \{x'\}) = \min\limits_{j=1,2}\min\limits_{0\leq i \leq s} V(P_i) \cup ST(G, (R\backslash label(x_1))\cup \{u_j\})$ where $P_i$ is the induced path obtained by removing the internal vertices of $P_{r_ir_{i+1}}$, $1\leq i \leq s-1$ from $label(x_1)$, $P_0$ and $P_s$ is obtained by removing the internal vertices of $P_{u_jr_{1}}$ and $P_{r_su_j}$ from $label(x_1)$, respectively.
\end{lemma}

\begin{proof}
We add either $u_1$ or $u_2$ to the terminal set to get the connected graph $T$. We list all the possibilities by adding $u_1$ to the terminal set and by adding $u_2$ to the terminal set, separately. Finally, we choose the minimum of all in order to get a minimum Steiner tree.
$\hfill \qed$
\end{proof}

\begin{lemma}
\label{sx1k/2+1pendcycle}
Let $label(x_1)=\{u_1,\ldots,u_k\}$ be the $(\frac{k}{2}+1)$-pendant cycle where $\{u_1,u_k\}\in E(G)$, $\{u_i,u_{i+1}\}\in E(G), 1 \leq i \leq k-1$,$deg_G(u_1)\geq 3$ and $deg_G(u_{\frac{k}{2}+1})\geq 3$. Let $\{r_1,\ldots,r_s\}\subseteq \{u_1,\ldots,u_k\}$ be the set of terminal vertices in $label(x_1)$. Then $x'$ is either $u_1$ or $u_{\frac{k}{2}+1}$ and $ST(G, (R \cap label(x_1))\cup \{x'\}) = \min\limits_{j=1,{\frac{k}{2}+1}}\min\limits_{0\leq i \leq s} V(P_i) \cup ST(G, (R\backslash label(x_1))\cup \{u_j\})$ where $P_i$ is the induced path obtained by removing the internal vertices of $P_{r_ir_{i+1}}$, $1\leq i \leq s-1$ from $label(x_1)$, $P_0$ and $P_s$ is obtained by removing the internal vertices of $P_{u_1r_{1}}$ and $P_{r_su_{\frac{k}{2}+1}}$ from $label(x_1)$, respectively.
\end{lemma}

\begin{proof}
The argument similar to \emph{Lemma \ref{sx12pendcycle}} establishes the claim. $\hfill \qed$
\end{proof}

Thus, we get a polynomial-time algorithm to find a minimum Steiner set using \emph{Lemmas \ref{sx1pendvertex}-\ref{sx1k/2+1pendcycle}}. Steiner tree can be obtained by finding a minimum spanning tree of the induced subgraph on  $ST(G,R)$. \\

%
%
%
%
%
%

%
%
%
%
%
%

\section{Structural and Algorithmic Results on $2K_2$-free graphs}
\label{2k2}

It is known from \cite{alan,white} that Steiner tree and dominating set are NP-Complete on $2K_2$-free graphs. In this section, we study subclasses of $2K_2$-free graphs where these two problems are polynomial-time solvable. Further, on such subclasses, we show that FVS and OCT are also polynomial-time solvable. To the best of our knowledge, this line of study has not been explored in the literature on these problems.

\subsection{$(2K_2, C_3, C_4)$-free graphs}

$(2K_2, C_3, C_4)$-free graphs form a proper subclass of $2K_2$-free graphs, where every induced cycle is of length 5. We observed the following structural properties and conclude that it is a trivial graph class.

\begin{theorem}
\label{c3c4}
If $G$ is a connected $(2K_2, C_3, C_4)$-free graph, then for any minimal vertex separator $S$ of $G$ satisfies the following properties:
\begin{itemize}
\item[(i)] $S$ is an independent set.
\item[(ii)] If $\mid S \mid > 1$, then $G\backslash S$ have exactly one trivial component.
\item[(iii)] If $G\backslash S$ has a non-trivial component, say $G_1$, then for every edge $\{u,v\} \in E(G_1)$, $(N_G(u) \cap S) \cap (N_G(v) \cap S) = \emptyset$ and $(N_G(u) \cap S) \cup (N_G(v) \cap S) = S$. i.e., For every vertex $x \in S$, $(N_G(x) \cap V(G_1))$ is an independent set.
\item[(iv)] Every vertex in a non-trivial component is adjacent to exactly one vertex in $S$.
\end{itemize}
\end{theorem}

\begin{proof}
\begin{itemize}
\item[(i)] On the contrary, assume that $S$ has at least one edge, say $\{x,y\}$. Let $G_i$ be  a trivial component in $G\backslash S$ and let $V(G_i) = \{w\}$. Since, $G$ is a $2K_2$-free graph, $\{w,x\}, \{w,y\} \in E(G)$ (by \emph{Theorem \ref{2k2pre}.(ii)}). Thus, $(w,x,y)$ forms an induced $C_3$, which is a contradiction to the definition of $G$. Hence, $S$ is an independent set.
\item[(ii)] On the contrary, assume that $G\backslash S$ has at least two trivial components, say $G_i$ and $G_j$. Let $V(G_i) = \{w_i\}$ and $V(G_j) = \{w_j\}$. Let $x,y$ be any two vertices in $S$. By $(i)$, $\{x,y\} \notin E(G)$ and by \emph{Theorem \ref{2k2pre}.(ii)}, $\{w_i,x\}$, $\{w_i, y\}$, $\{w_j, x\}$, $\{w_j, y\} \in E(G)$. Thus, $(w_i,x,w_j,y)$ forms an induced $C_4$, which is a contradiction to the definition of $G$. Hence, $G\backslash S$ have exactly one trivial component if $\mid S \mid >1$.
\item[(iii)] By \emph{Theorem \ref{2k2pre}.(iii)}, every edge $\{u,v\} \in E(G_1)$ is universal to $S$, thus, $(N_G(u) \cap S) \cup (N_G(v) \cap S) = S$. Moreover, if $(N_G(u) \cap S) \cap (N_G(v) \cap S) \neq \emptyset$, then every vertex in $(N_G(u) \cap S) \cap (N_G(v) \cap S)$ forms an induced $C_3$ together with $u$ and $v$. Hence, $(N_G(u) \cap S) \cap (N_G(v) \cap S) = \emptyset$.
\item[(iv)] On the contrary, assume that exist a vertex $v$ in a non-trivial component such that $(N_G(v) \cap S) = \{x_1, x_2, \ldots, x_p\},$ $p \geq 2$. By $(ii)$, there exist a trivial component in $G \backslash S$, say $G_2$. Let $V(G_2) = \{w\}$. Therefore, $(v, x_1, x_2, w)$ forms an induced $C_4$, which is a contradiction to the definition of $G$. $\hfill \qed$
\end{itemize}
\end{proof}

\begin{corollary}
If $G$ is a connected $(2K_2, C_3, C_4)$-free graph, then $G$ is either a tree or $C_5$.
\end{corollary}

\begin{proof}
From \emph{Theorem \ref{c3c4}}, we can observe that the only possible structure of a non-trivial component after the removal of any minimal vertex separator from $G$ is $K_2$ and $\mid S \mid \leq 2$. Further, if $\mid S \mid = 1$, then the graph is $(2K_2,cycle)$-free. If $\mid S \mid = 2$ and if $G\backslash S$ has a non-trivial component, then the graph is an induced $C_5$.
$\hfill \qed$
\end{proof}

 Thus, FVS, OCT, Steiner tree problem and a dominating set can be solved in $O(1)$ time when the input is restricted to $(2K_2, C_3, C_4)$-free graphs.

\subsection{$(2K_2, C_3, C_5)$-free graphs}
$(2K_2, C_3, C_5)$-free graphs are $2K_2$-free graphs which are either acyclic or every induced cycle is of length 4. Further, these graphs are $2K_2$-free chordal bipartite graphs. We shall study this graph class from MVS perspective.
\begin{theorem}
\label{c3c5}
If $G$ is a connected $(2K_2, C_3, C_5)$-free graph, then for any minimal vertex separator $S$ of $G$ satisfies the following properties:
\begin{itemize}
\item[(i)] $S$ is an independent set.
\item[(ii)] If $G\backslash S$ has a non-trivial component, say $G_1$, then for every vertex $x \in S$, $(N_G(x) \cap V(G_1))$ is an independent set.
\item[(iii)] For every edge $\{u,v\}$ in a non-trivial component $G_1$ of $G\backslash S$, $u$ is universal to $S$ and $(N_G(v) \cap S) = \emptyset$.
\item[(iv)] Let $T$ be the set of all vertices in the trivial components of $G\backslash S$. Then the graph induced on the vertex set $T\cup S$ is a complete bipartite graph.
\item[(v)] Let $U$ and $U'$ be the set of all vertices in a non-trivial component which are universal  and non-universal to $S$, respectively. Then, there exists a vertex $u \in U$ such that $u$ is universal to $U'$.
\end{itemize}
\end{theorem}

\begin{proof}
\begin{itemize}
\item[(i)] The argument is similar to the proof in \emph{Theorem \ref{c3c4}.(i)}.
\item[(ii)] The argument is similar to the proof in \emph{Theorem \ref{c3c4}.(iii)}.
\item[(iii)] On the contrary, assume that there exists an edge $\{u,v\} \in E(G_1)$ such that $S \not\subseteq N_G(u)$, $(N_G(v) \cap S) \neq \emptyset$ and $(N_G(u) \cap S) \neq \emptyset$. Since, $G$ is $2K_2$-free graph, $(N_G(u) \cap S) \cup (N_G(v) \cap S) = S$ and there exists a trivial component in $G\backslash S$, say $G_2$. Let $V(G_2) = \{w\}$. By our assumption, $u$ is adjacent to some vertex in $S$, say $x$ and $v$ is adjacent to some vertex in $S$, say $y$, such that $x \neq y$. Thus, $(u,v, y, w, x)$ forms an induced $C_5$, which is a contradiction to the definition of $G$.
\item[(iv)] This is true by the fact that $S$ is independent and every trivial component is universal to $S$.
\item[(v)] By (iii), $G_1$ is a bipartite graph where $U$ and $U'$ are the independent sets. Let us prove the statement by mathematical induction on the cardinality of $U$. \\
\textit{Base Case:} Since $G_1$ is connected, the statement is true for $\vert U \vert =1$.\\
\textit{Hypothesis:} Assume that the statement is true for $\vert U \vert = s, ~s \geq 1$.\\
\textit{Induction Step:} Let $\vert U \vert = s+1, ~s \geq 1$.\\
For some $u \in U$, the graph $G_1\backslash \{u\}$ has a vertex $v \in U$ universal to $U'$, by the hypothesis. If $N_{G_1}(u) \subset N_{G_1}(v)$, then there is nothing to prove. W.l.o.g. assume that $N_{G_1}(u) \backslash N_{G_1}(v)\neq \emptyset$. For arbitrary $x \in N_{G_1}(u) \backslash N_{G_1}(v)$. If $\{v,x\} \in E(G)$, then $v$ is the required vertex which is universal to $U'$. If $\{v,x\} \notin E(G)$, then $u$ is the required vertex which is universal to $U'$. $\hfill \qed$
\end{itemize}
\end{proof}

Although, it is known that the problem of finding a minimum feedback vertex set in chordal bipartite graphs, a super class of $2K_2$-free chordal bipartite graphs, is polynomial time solvable \cite{kloks}, using the above observation we provide a different approach for this problem in $(2K_2, C_3, C_5)$-free graph. Moreover, our approach takes linear time in terms of the input size. Also, it is easy to see that FVS is precisely ECT. 

\begin{theorem}
\label{fvsc3c5}
Let $G$ be a connected $(2K_2, C_3, C_5)$-free graph and $S$ be any minimal vertex separator of $G$, then the cardinality of any minimum feedback vertex set $F$ is
\begin{itemize}
\item[(i)] $min \{\mid S \mid -1, \mid T \mid - 1\}$, if $G\backslash S$ has only trivial components, and $T$ is the set of all trivial components in $G\backslash S$.
\item[(ii)] $min \{ \mid S \mid, \mid U \mid + (\mid T \mid -1)\}$, if $G\backslash S$ has a non-trivial component $G_1$, which is cycle-free, and $U$ is the set of all vertices in $G_1$ which are universal to $S$.
\item[(iii)] $min \{ \mid U \mid + (\mid  T \mid -1), (\mid U \mid -1) + (\mid S \mid -1)\}$, if $G\backslash S$ has a non-trivial component $G_1$ and $G_1$ has at least one cycle.
\end{itemize}
\end{theorem}

\begin{proof}
\begin{itemize}
\item[(i)] If $G$ is a cycle-free graph, then either $\mid S \mid = 1$ or $\mid T \mid = 1$. Thus, $F = \emptyset$, which is minimum. Without loss of generality, assume that $G$ has at least one cycle and $G \backslash S$ has only trivial components, say $G_1, G_2, \ldots, G_l$, $l \geq 2$. By our assumption, $\mid S \mid \geq 2$ and by \emph{Theorem \ref{c3c5}}, $S$ is an independent set. Let $V(G_i)=\{u_i\}$. Clearly, $G \backslash F$ results in a forest, where $F$ consists of $\mid S \mid$ - 1 vertices from $S$ and $\mid T \mid$ - 1 vertices from $T$. Now, our claim is to prove the set $F$ is minimum.
\begin{itemize}
\item[$\bullet$] $F = min \{\mid S \mid -1, \mid T \mid - 1\} = \mid S \mid - 1$ \\
On the contrary, assume that $F$ is not minimum, then the removal of $S'$ vertices from $G$ results in a forest, where $S' <  \mid S \mid - 1$. I.e., $S$ has at least two vertices in $G \backslash F$, say $x, y \in S$. Clearly, $(u_1, x, u_2, y)$ forms an induced $C_4$, which is a contradiction to the definition of $F$. 
\item[$\bullet$]  $F = min \{\mid S \mid -1, \mid T \mid - 1\} = \mid T \mid - 1$ \\
On the contrary, assume that $F$ is not minimum, then the removal of $T'$ vertices from $G$ results in a forest, where $T' <  \mid T \mid - 1$. I.e., $T$ has at least two vertices in $G \backslash F$, say $u_1, u_2 \in T$. Let $x$ and $y$ be any two vertices in $S$. Clearly, $(u_1, x, u_2, y)$ forms an induced $C_4$, which is a contradiction to the definition of $F$. 
\end{itemize}
Hence, $F$ is a minimum FVS if $G\backslash S$ has only trivial components.
\item[(ii)] All possible structures of $G_1$ is given in \emph{Figure \ref{fig:g123}}. From the structures of $G_1$, it is clear that $F$ is a minimum FVS. It follows from \emph{Theorem \ref{c3c5}} that no more structures of $G_1$ are possible. 

\begin{figure}[h]
\centering
\includegraphics[scale=0.35]{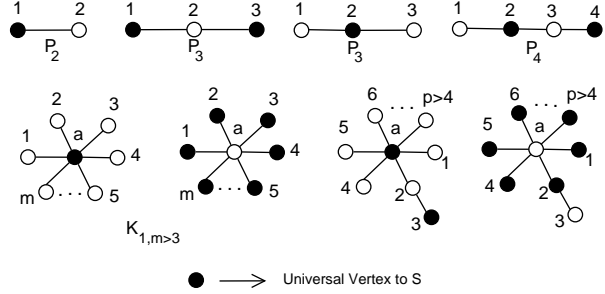}
\caption{All Possible structures of $G_1$ when $G_1$ is cycle-free}
\label{fig:g123}
\end{figure}

\item[(iii)] We prove this case separately for $\mid S \mid = 1$ and $\mid S \mid >1$.
\begin{itemize}
\item[$\bullet$] $\mid S \mid = 1$ and let $S = \{x\}$. \\
It is clear that, every cycle of $G$ lies in $G_1$. Thus, $F = min \{ \mid U \mid + (\mid  T \mid -1), (\mid U \mid -1) + (\mid S \mid -1)\} = \mid U \mid -1$ and the removal of $\mid U \mid -1$ vertices from $U$ results in a forest. Now, our claim is to prove that $F$ is minimum. On the contrary, assume that removing at most $\mid U \mid$ - 2 vertices from $U$ results in a forest. I.e., $G\backslash F$ has at least two vertices in $U$, say $v,w \in U$. Since, $G$ is $2K_2$-free $\mid P_{vw} \mid \leq 4$. Note that, $\mid P_{vw} \mid \neq 2$ because every edge in $G_1$ is between an universal vertex and a non-universal vertex in $G_1$, by \emph{Theorem \ref{c3c5}.(iii)}. Similarly, $\mid P_{vw} \mid \neq 4$. Thus, the only possibility is $\mid P_{vw} \mid = 3$. Therefore, $(P_{vw},x)$ forms an induced $C_4$, which is a contradiction to $F$.    
\item[$\bullet$] $\mid S \mid > 1$. $S$ has at least two vertices, say $x, y \in S$. Our claim is to prove that $S$ is minimum.
\begin{itemize}
\item[-] $F = min \{ \mid U \mid + (\mid  T \mid -1), (\mid U \mid -1) + (\mid S \mid -1)\} = \mid U \mid + (\mid  T \mid -1)$ \\
On the contrary, assume that for some $a \in T$ there exists a set $M\subset (U \cup (T \backslash \{a\}))$ such that $\vert M \vert < F$ and $G\backslash M$ is a forest. Let $v \in U-M$. Then $(a,x,v,y)$ forms an induced $C_4$, which is a contradiction. Let $b \in T-M$ and $b\neq a$. Then $(a,x,b,y)$ forms an induced $C_4$, which is a contradiction.

\item[-] $F = min \{ \mid U \mid + (\mid  T \mid -1), (\mid U \mid -1) + (\mid S \mid -1)\} = (\mid U \mid -1) + (\mid S \mid -1)$ \\
On the contrary, assume that for some $v \in U$ there exists a set $M\subset (U \backslash \{v\}) \cup (S \backslash \{x\})$ such that $\vert M \vert < F$ and $G\backslash M$ is a forest. Let $w \in U-M$ and $w \neq v$. Then $(P_{vw},y)$ forms an induced $C_4$, which is a contradiction. Let $y \in S-M$. Then for any $a \in T$, $(a,x,v,y)$ forms an induced $C_4$, which is a contradiction.
\end{itemize}
\end{itemize}
\end{itemize}
From all the above cases, it is proved that $F$ is a minimum FVS. Hence, the theorem. $\hfill \qed$
\end{proof}

\begin{theorem}
\label{stc3c5}
Let $G$ be a connected $(2K_2, C_3, C_5)$-free graph, $R \subseteq V(G)$ be the terminal set of $G$ and $S$ be any MVS of $G$. Let $T$ be the set of all trivial components in $G\backslash S$, $U$ and $U'$ be the set of universal and non-universal vertices in a non-trivial component of $G\backslash S$, respectively. If $R$ is connected, then the Steiner tree $ST(G,R)$ is the graph induced on the vertex set $R$. If $R$ is not connected, then the Steiner tree $ST(G,R)$ is the graph induced on the vertex set

\begin{itemize}
\item $R \cup \{x\}$, for some $x \in S$, if $R\backslash T$ is connected or when $R$ is the subset of $T$ or $U$ or ($T\cup U$).
\item $R \cup \{a\}$, for some $a \in T$, when $R$ is the subset of $S$
\item $R \cup \{v\}$, where $v \in U$ is universal to $U'$, when $R$ is the subset of $U'$ or ($S\cup U'$) or ($U\cup U'$) or ($T\cup S \cup U'$) or ($S\cup U \cup U'$) or ($T \cup S\cup U \cup U'$).
\item $R \cup \{v\}\cup \{x\}$, for some $x \in S$ and a vertex $v \in U$ universal to $U'$, if $R\backslash T$ is connected or $R \subseteq T\cup U'$.
\end{itemize}
\end{theorem}

\begin{proof}
Trivially follows from \emph{Theorem \ref{c3c5}}. $\hfill \qed$
\end{proof}

\begin{theorem}
\label{dsc3c5}
Let $G$ be a connected $(2K_2, C_3, C_5)$-free graph and $S$ be any minimal vertex separator of $G$. Let $T$ be the set of all trivial components in $G\backslash S$, $U$ and $U'$ be the set of universal and non-universal vertices in a non-trivial component of $G\backslash S$, respectively. If $G\backslash S$ has only trivial components, then the dominating set is $\{x,a\}$, for some $x \in S$ and $a \in T$ when $\vert S \vert \geq 2$, and the dominating set is $S$ when $\vert S \vert =1$. If $G\backslash S$ has a non-trivial component, then the dominating set is $\{x,u\}$, for some $x \in S$ and $u \in U$ is universal to $U'$.
\end{theorem}

\begin{proof}
Trivially follows from \emph{Theorem \ref{c3c5}}. $\hfill \qed$
\end{proof}

\emph{Theorem \ref{fvsc3c5}}, \emph{Theorem \ref{stc3c5}} and \emph{Theorem \ref{dsc3c5}} naturally yields an algorithm to find a minimum FVS, Steiner tree and dominating set, respectively, in $O(n)$ time, which is linear in the input size.

\subsection{$(2K_2, C_4, C_5)$-free graphs}

$(2K_2, C_4, C_5)$-free graphs are $2K_2$-free graphs where every induced cycle is of length 3. This graphs can also be called as $2K_2$-free chordal graphs. Note that $2K_2$-free chordal graphs are known as split graphs. We know that the structural of any minimal $(a,b)$-vertex separator in chordal graphs is a clique. It is important to highlight that, the feedback vertex set problem is solvable in polynomial time, for chordal graphs \cite{cor}, a superclass of split graphs.

\begin{theorem}
\label{fvsc4c5}
Let $G$ be a connected $(2K_2, C_4, C_5)$-free graph and $S$ be any MVS of $G$, then a minimum FVS $FVS(G)$ is

\begin{itemize}
\item[(i)] $ V(G)\backslash \{x,y\}$, if $G$ is a complete graph, for some $x,y \in V(G)$.
\item[(ii)] $ S  \backslash \{v\}$, for some $v\in S$, if $G\backslash S$ has only trivial components.
\item[(iii)] $G\backslash S$ has a non-trivial component $G_1$ and $G_1$ is a tree. If there exist a vertex $v \in S$ such that $\mid N_G(v) \cap V(G_1) \mid = 1$, then $FVS(G)=S\backslash \{v\}$. If for every vertex $v \in S$, $\mid N_G(v) \cap V(G_1) \mid \geq 2$, then $FVS(G)=S$.

\item[(iv)] $\min\limits_{j \in S} \{ \mid S\backslash \{j\} \mid + FVS(G_1 \cup \{j\})\}$, if $G\backslash S$ has a non-trivial component $G_1$ and $G_1$ has at least one cycle.
\end{itemize}
\end{theorem}

\begin{proof}
\begin{itemize}
\item[(i)] The proof is obvious from the definition of complete graphs.
\item[(ii)] Since, $S$ is a clique, we have to remove at least $\mid S \mid -2$ vertices from $S$. Assume that the remaining edge in $S$ is $\{u,v\}$, after  the removal of $\mid S\mid -2$ vertices. We know that $G\backslash S$ has at least two components and given that every component in $G\backslash S$ is a trivial component. Thus, we have to remove any one vertex from $\{u,v\}$ such that all cycles formed between trivial components and an edge $\{u,v\}$ are removed.
\item[(iii)] By (ii), it is clear that we have to remove at least $\mid S \mid-1$ vertices from $S$. If there exists a vertex, $v$, in $S$ whose neighborhood in a non-trivial component is a singleton set, then the removal of $M = S\backslash \{v\}$ from $G$ creates a forest and thus, $FVS(G)=M$. If every vertex in $S$ has more than one vertex in $G_1$ as its neighbor, then $u$ forms at least one cycle along with $G_1$, thus, $FVS(G)= S$.
\item[(iv)] We enumerate all possible feedback vertex set in $S\cup G_1$, whose removal from $G$ results in a forest, and we choose the minimum among them. $\hfill \qed$
\end{itemize}
\end{proof}                                                                                                                                                                                                                                                                                                                                                                                                                                                                                                                                                                                                                                                                                                                                                                                                                                                                                                                                                                                                                                                                                                                                                                                                                                                                                                                         
\emph{Theorem \ref{fvsc4c5}} naturally yields an algorithm to find a minimum FVS in $O(n^2\delta)$ time. It is important to highlight that, the feedback vertex set problem is solvable in polynomial time, $O(n^5)$, for chordal graphs \cite{cor}, a superclass of split graphs.

\subsection{$(2K_2, C_3)$-free graphs}
$(2K_2, C_3)$-free graphs are $2K_2$-free graphs where every induced cycle is of length $4$ or 5. A structural observation is given below:

\begin{definition}
Let $G$ be a connected graph and $S$ be a minimal vertex separator for $G$. Let $G_1, \ldots, G_s$ be the connected components of $G\backslash S$. For some $u,v \in V(G_i)$, $P_{uv}^{i}$ denotes the shortest path between $u$ and $v$ in a graph $G$ such that all internal vertices belongs to $V(G_i)$.
\end{definition}

\begin{theorem}
\label{c3}
If $G$ is a connected $(2K_2, C_3)$-free graph, then for any minimal vertex separator $S$ of $G$ satisfies the following properties:

\begin{itemize}
\item[(i)] $S$ is an independent set.
\item[(ii)] If $G\backslash S$ has a non-trivial component $G_1$, then for  every vertex $x \in S$, $(N_G(x) \cap V(G_1))$ is an independent set. Moreover, $\mid P_{uv}^{1} \mid = 3$, for all $u,v \in (N_G(x) \cap V(G_1))$. 
\item[(iii)] If $\mid S \mid \geq 2$ and $G\backslash S$ has a non-trivial component $G_1$, then $G_1\backslash M$ is $P_4$-free, where $M = \{v \in V(G_1) \mid N_G(v) \cap S = \phi\}$. Moreover, $M$ is independent and there exist a unique vertex $u\in G_1\backslash M$ such that $u$ is universal to $M$.
\item[(iv)] If $G\backslash S$ has a non-trivial component, say $G_1$, then $G_1$ is $C_5$-free. Further, the graph induced on $G_1\cup S$ is $C_5$-free.
\end{itemize}
\end{theorem}

\begin{proof}
\begin{itemize}
\item[(i)] The argument is similar to the proof in \emph{Theorem \ref{c3c4}.(i)}.
\item[(ii)] The argument is similar to the proof in \emph{Theorem \ref{c3c4}.(iii)}. Let $u$ and $v$ be any two vertices in $(N_G(x) \cap V(G_1))$. Our claim is to prove that $\mid P_{uv}^{1} \mid = 3$. On the contrary, assume that $\mid P_{uv}^{1} \mid = 4$ (Since, $G$ is $2K_2$-free, $\mid P_{uv}^{1} \mid \ngeq 5$), say $P_{uv}^{1} = (u,w,s,v)$. We know that in a $2K_2$-free graph, every edge in a non-trivial component is universal to $S$. Thus, either $\{w,x\} \in E(G)$ or $\{s,x\} \in E(G)$. If $\{w,x\} \in E(G)$, then $(u,w,x)$ forms an $C_3$ or if $\{s,x\} \in E(G)$, then $(x,s,v)$ forms an $C_3$, which is a contradiction to the definition of $G$.
\item[(iii)] On the contrary, assume that $G_1\backslash M$ has an induced $P_4$, say $P_4 = (u,v,w,s)$. Choose any two vertices $x$ and $y$ from $S$. Either $\{x,u\}, \{x,w\}, \{y,v\}, \{y,s\} \in E(G)$, where $P = (x,u,v,y,s)$ forms an induced $P_5$ ($P$ is induced by $(ii)$) or $\{y,u\}, \{y,w\}, \{x,v\}, \{x,s\} \in E(G)$, where $P' = (y,u,v,x,s)$ forms an induced $P_5$ ($P'$ is induced by $(ii)$), which is a contradiction to the definition of $G$. $M$ is independent because of the fact every edge in $G_1$ is universal to $S$. The existence of universal vertex to $M$ in $G_1\backslash M$ is true by the fact $G$ is $2K_2$-free and it is unique by (ii).
\item[(iv)] On the contrary, assume that $G_1$ has an induced $C_5 = (u_1,u_2, u_3,u_4,u_5)$. Choose a vertex $x\in S$. Since, every edge in $G_1$ is universal to $S$, any one of the following is true:
\begin{itemize}
\item[$\bullet$] $\{u_1,x\}, \{u_3,x\}, \{u_5,x\} \in E(G)$, then $(u_1,u_5,x)$ forms a $C_3$.

\item[$\bullet$] $\{u_2,x\}, \{u_4,x\} \in E(G)$, then the edge $\{u_1,u_5\}$ is not universal to $S$.
\end{itemize}
\end{itemize}
\noindent Both contradicts the definition of $G$. Since $S$ is independent, the graph induced on $G_1 \cup S$ is also $C_5$-free.  $\hfill \qed$
\end{proof}

\emph{Theorem \ref{c3}} naturally yields an algorithm to find the FVS, which is described as follows. Finding a FVS in a $(2K_2,C_3)$-free graph is same as finding a FVS in $ (S\cup G_1)$, say $A$, and in $G\backslash A$, which is a recursive call and the recursion bottoms out when it returns a bipartite graph, $(2K_2,C_3,C_5)$-free graph. This can be done in polynomial time.

%
%

\begin{theorem}
\label{stc3}
Let $G$ be a connected $(2K_2, C_3)$-free graph, $R \subseteq V(G)$ be the terminal set of $G$ and $S$ be any minimal vertex separator of $G$. Let $T$ be the set of all trivial components in $G\backslash S$. If $R$ is connected, then the Steiner tree $ST(G,R)$ is the graph induced on the vertex set $R$. If $R$ is not connected, then the Steiner tree $ST(G,R)$ is the graph induced on the vertex set

\begin{itemize}
\item $R \cup \{x\}$, for some $x \in S$, if $R\subseteq T$.
\item $R \cup \{a\}$, for some $a \in T$, if $R\subseteq S$.
\item $\min\limits_{\forall ~x_i \in S}\{ST([S\cup V(G_1)], R\cup \{x_i\})\}$, if $R\subseteq (T\cup G_1)$.
\item $ST([S\cup V(G_1)], R)$, if $R$ is the subset of $G_1$ or $(S\cup G_1)$ or $(T\cup S \cup G_1)$.
\end{itemize}
\end{theorem}

\begin{proof}
Trivially follows from \emph{Theorem \ref{c3}}. $\hfill \qed$
\end{proof}

\begin{theorem}
\label{dsc3}
Let $G$ be a connected $(2K_2, C_3)$-free graph and $S$ be any minimal vertex separator of $G$. Let $T$ be the set of all trivial components in $G\backslash S$. If $G\backslash S$ has only trivial components, then the dominating set is $\{x,a\}$, for some $x \in S$ and $a \in T$ when $\vert S \vert \geq 2$, and the dominating set is $S$ when $\vert S \vert =1$. If $G\backslash S$ has a non-trivial component, then the dominating set is $\min\limits_{\forall ~ x_i \in S}\{\{x_i\}\cup \{a\} \cup D_i \}$, where $a \in T$ and $D_i$ is the dominating set of the graph induced on $(S\cup V(G_1)) \backslash (x_i \cup N_G(x_i))$, which is $2K_2$-free chordal bipartite graph.
\end{theorem}

\begin{proof}
Trivially follows from \emph{Theorem \ref{c3}}. $\hfill \qed$
\end{proof}

 It is easy to see that the \emph{Theorem \ref{stc3}} and \emph{Theorem \ref{dsc3}} yields a linear time algorithm to find a Steiner tree and dominating set, respectively.


\subsection{$(2K_2, C_4)$-free graphs}
$(2K_2, C_4)$-free graphs are $2K_2$-free graphs where every induced cycle is of length 3 or 5. The structural observations for this graph class are as follows:
\begin{theorem}
\label{c4}
If $G$ is a connected $(2K_2, C_4)$-free graph, then for any minimal vertex separator $S$ of $G$ satisfies the following properties:
\begin{itemize}
\item[(i)] $S$ is connected except if $G$ is an induced $C_5$ or $K_{1,m}, m \geq 2$.
\item[(ii)] $S$ is connected and has a non-trivial component, $G_1$, in $G\backslash S$. If a vertex $x \in V(G_1)$ is adjacent to a vertex $u \in S$, then $(N_G(u)\cap S) \subseteq N_G(x)$.
\item[(iii)] If $S$ is not a clique, then $G\backslash S$ has exactly one trivial component. Moreover, every vertex in a non-trivial component of $G\backslash S$ is not universal to any non-adjacent pair of vertices in $S$.
\item[(iv)] If $S$ is not a clique, then the only possibility of a non-trivial component of $G\backslash S$ is $K_2$.
\item[(v)] The size of the maximum independent set of the graph induced on $S$ is at most two.
\item[(vi)] $S$ contains neither $P_4$ nor $K_{1,m}$, $m \geq 3$.
\end{itemize}
\end{theorem}

\begin{proof}
\begin{itemize}
\item[(i)] On the contrary, assume that $G[S]$ has at least two components. Choose two vertices $x$ and $y$ from different components of $G[S]$. If $G\backslash S$ has only trivial components, then $x$, $y$ and any two trivial components from $G\backslash S$ forms $C_4$, which is a contradiction. If $G\backslash S$ has a non-trivial component, $G_1$, then choose an edge $\{u,v\} \in E(G_1)$. If $\mid S \mid = 2$, then either $u$ is universal to $S$ or $v$ is universal to $S$. W.l.o.g, assume that $u$ is universal to $S$. Thus, $u,x,y$ and a trivial component in $G\backslash S$ forms a $C_4$, which is a contradiction to the definition of $G$. If $\mid S \mid \geq 3$, then either $\mid N_G(u) \cap S \mid \geq 2$ or $\mid N_G(v) \cap S \mid \geq 2$. W.l.o.g, assume that $\mid N_G(u) \cap S \mid \geq 2$. Let $x,y \in (N_G(u) \cap S)$. Thus, $u,x,y$ and a trivial component in $G\backslash S$ forms a $C_4$, which is a contradiction to the definition of $G$.
\item[(ii)] On the contrary, assume that $\{x,v\} \notin E(G)$ for some $v \in (N_G(u)\cap S)$. Since, $G$ is $2K_2$-free and $G_1$ is a non-trivial component, there exists a vertex $y \in G_1$ such that $\{x,y\}, \{y,v\} \in E(G)$. Thus, $(x,u,v,y)$ forms an induced $C_4$, which is a contradiction.
\item[(iii)] On the contrary, assume that $G\backslash S$ has more than one trivial component. Let $\{u\}$ and $\{v\}$ be any two trivial components in $G\backslash S$. Since $S$ is not a clique, $S$ contains a $P_3=(x,y,z)$. Since, $G$ is a $2K_2$-free graph, $\{u,x\},\{u,z\},\{v,x\},\{v,z\} \in E(G)$. Thus, $(u,x,v,z)$ forms an induced $C_4$, which is a contradiction to the definition of $G$. Moreover, if there exists a vertex, $u$, in a non-trivial component of $G\backslash S$ is universal to some non-adjacent pair $(x,z)$ in $S$ and if $\{v\}$ is a trivial component of $G\backslash S$, then $(u,x,v,z)$ forms an induced $C_4$, which is a contradiction.
\item[(iv)] Since $S$ is not a clique, $S$ contains a $P_3$, say $P_3 = (x,y,z)$. On the contrary, assume that the non-trivial component of $G\backslash S$, $G_1$, contains either $K_3=(u,v,w)$ or $P_3=(u,v,w)$. Consider an edge $\{u,v\}$, since every edge in $G_1$ is universal to $S$, either $\{u,x\},\{u,y\}, \{v,y\},\{v,z\} \in E(G)$ or $\{v,x\},\{v,y\}, \{u,y\},\{u,z\} \in E(G)$. W.l.o.g, assume that, $\{u,x\},\{u,y\}, \{v,y\},\{v,z\} \in E(G)$. Now, consider the edge $\{v,w\}$, since, $\{v,y\}, \{v,z\} \in E(G)$ either $\{x,v\} \in E(G)$ or $\{x,w\} \in E(G)$. By (iii), $\{x,v\} \notin E(G)$. Thus, the only possibility is $\{x,w\} \in E(G)$. If $(u,v,w)$ is a path, then $(x,u,v,w)$ forms an induced $C_4$, which is a contradiction. If $(u,v,w)$ is $K_3$, then consider the edge $\{u,w\}$, either $\{u,z\} \in E(G)$ or $\{w,z\} \in E(G)$. By (iii), both $\{u,z\}, \{w,z\} \notin E(G)$. Thus, the edge $\{u,w\}$ is not universal to $S$, which is a contradiction.

\item[(v)] On the contrary, assume that there exists at least three mutually independent vertices, say $\{x,y,z\}$ in $S$. It is clear that $S$ is not a clique, therefore by (iii) and (iv), there exists a trivial component and a non-trivial component, i.e., a $K_2 = \{u,v\}$, in $G\backslash S$.  By (iii), $u (v)$ can be adjacent to at most one vertex in  $\{x,y,z\}\}$. W.l.o.g, assume that $\{u,x\}, \{v,y\} \in E(G)$. This implies, neither $u$ is adjacent to the vertex $z$ nor $v$ is adjacent to the vertex $z$, which is a contradiction to \emph{Theorem 1.(iii)}. 

\item[(vi)] On the contrary, $S$ contains either $P_4$ or $K_{1,m}$, $m \geq 3$. If $S$ contains a $P_4=(x,y,z,s)$: By (iii), $G\backslash S$ has exactly one trivial component and a non-trivial component $K_2$, say $G_1=\{u,v\}$ (by (iv)). By (iii), either $\{u,x\},\{u,y\} \in E(G)$ or $\{u,z\},\{u,s\} \in E(G)$. W.l.o.g, assume that $\{u,x\},\{u,y\} \in E(G)$. Since, $G$ is $2K_2$-free, every edge in $G_1$ is universal to $S$. Therefore, $\{v,z\},\{v,s\} \in E(G)$. By (iii), $\{u,z\}, \{v,y\} \notin E(G)$. Hence, $(u,y,z,v)$ forms an induced $C_4$, which is a contradiction. Proof for $S$ does not contains $K_{1,m}, m \geq 3$ directly follows from (v). $\hfill \qed$

\end{itemize}
\end{proof}

By the \emph{Theorem \ref{c4}}, it is clear that $S$ is $C_5$-free. Hence, the feedback vertex set in a $(2K_2,C_4)$-free graph can be determined as follows:

\begin{theorem}
\label{fvsc4}
Let $G$ be a connected $(2K_2, C_4)$-free graph and $S$ be any minimal vertex separator of $G$, then the cardinality of a minimum FVS, $F$, is equal to

\begin{itemize}
\item[(i)] $\mid G \backslash \{i,j\}\mid $, if $G$ is a complete graph, for some $i,j \in V(G)$.

\item[(ii)] $\mid S\backslash \{j\} \mid$, if $S$ is a independent set, for some $j \in S$.

\item[(iii)] $\mid S\backslash \{j\}\mid$, if $S$ is neither clique nor a independent set, for some $j \in S$.

\item[(iv)] $\min\limits_{j \in S} \{ \mid S\backslash \{j\} \mid + \mid FVS(G_1 \cup \{j\})\mid\}$, if $S$ is a clique, where $G_1$ is a non-trivial component in $G\backslash S$.
\end{itemize}
\end{theorem}

\begin{proof}
\begin{itemize}
\item[(i)] The proof is obvious from the definition of complete graphs.

\item[(ii)] From \emph{Theorem \ref{c4}.(i)}, it is clear that $S$ is independent only when $G = C_5$. Also, $\mid S \mid = 2$ and $\mid S \mid -1$ says that $FVS(G) = \{j\}$, for some $j\in V(G)$ i.e., $\mid FVS(G) \mid = 1$. Thus, the removal of a vertex from $G$ makes $G$ a tree and it is minimum.

\item[(iii)] We know that, $S$ is a split graph. Since $S$ is not complete, $G\backslash S$ has a non-trivial component, $G_1=K_2= \{u,v\}$ and a trivial component, $G_2 = \{w\}$. Thus, finding $FVS(G)$ is equivalent to finding $W=FVS(S)$ and $FVS((S\backslash W) \cup G_1 \cup G_2)$. The possible structures of $S\backslash W$ are (a) $2K_1$ (b) $K_1 \cup K_2$ (c) $K_1 \cup P_3$ (d) $K_2$ (e) $P_3$ (by \emph{Theorem \ref{c4}}). In all the cases, we are forced to pick exactly $(S\backslash W)\backslash \{i\}$, for some $i \in S\backslash W$, vertices. Thus, $FVS(G) = S\backslash \{j\}$, for some $j \in S$. 
%

\item[(iv)] Since, $S$ is a clique and $G\backslash S$ has at least one trivial component, $FVS(G)$ contains at least $S\backslash \{j\}$, for some $j \in S$, vertices. If $G\backslash S$ has a non-trivial component, $G_1$, then we are forced to find $FVS(G_1 \cup \{j\})$ in order to compute $FVS(G)$. Thus, $FVS(G) = \min\limits_{j \in S} \{ \mid S\backslash \{j\} \mid + \mid FVS(G_1 \cup \{j\})\mid\}$. It is minimum because we are varying $j$ for all vertices in $S$ and picking up the minimum. $\hfill \qed$
\end{itemize}
\end{proof}

\noindent \textbf{Remark:}
Minimum dominating set and Steiner tree problem are NP-Complete restricted to $2K_2$-free graphs \cite{white}. In this paper, we have identified two non-trivial subclasses of $2K_2$-free graphs where these problems are polynomial time solvable.

\section{Applications}

In this section, we consider the complexity of connected dominating set and connected FVS using the results presented in \emph{Sections \ref{sec::algs}-\ref{2k2}}. It is interesting to observe that every minimum connected dominating set contains a minimum dominating set as a vertex subset. It is natural to ask, can we use a minimum dominating set as a terminal set and call Steiner tree algorithm as a black box to get a minimum connected dominating set. Surprisingly, this observation holds good for $SC_k$ graphs and subclasses of $2K_2$-free graphs. A similar observation is true for connected vertex cover and connected FVS. Further, maximum leaf spanning tree problem is also polynomial time solvable restricted to $SC_k$ and subclasses of $2K_2$-free graphs. Due to page constraint, the proof details are missing in this paper.


\end{document}